\documentclass[11pt,a4paper]{article}
\usepackage[utf8]{inputenc}
\usepackage{amsmath,amsthm}
\usepackage{amsfonts}
\usepackage{amssymb}
\usepackage[savegames]{xcolor}

\usepackage{tikz}
\usepackage[warn]{textcomp}

\usepackage{hyperref}
\usepackage{cases}
\hypersetup{colorlinks=true, linkcolor=blue, urlcolor=blue, 
  pdftitle={}, pdfauthor={}}

\title{Generating Functions of Some Families of
  Directed Uniform Hypergraphs}
\author{Rasendrahasina
  Vonjy\thanks{ENS -- Universit\'e d'Antananarivo,
    Madagascar:\href{mailto:rasendrahasina@gmail.com}{rasendrahasina@gmail.com}. },  
  \and Ravelomanana Vlady\thanks{
    IRIF -- UMR CNRS 8243 -- Universit\'e de Paris,
    France:\href{mailto:vlad@irif.fr}{vlad@irif.fr}.}} 
\date{}

\newtheorem{thm}{Theorem}
\newtheorem{lem}{Lemma}
\newtheorem{remark}{Remark}
\newtheorem{dfn}{Definition}

\usetikzlibrary{calc} 
\usetikzlibrary{shadows} 
\usetikzlibrary{arrows} 
\usetikzlibrary{patterns} 
\usetikzlibrary{matrix}

\usetikzlibrary{calc} 
\usetikzlibrary{shadows} 
\usetikzlibrary{arrows} 
\usetikzlibrary{patterns} 
\usetikzlibrary{matrix}

\definecolor{bblue}{rgb}{0.2, 0.4, 0.8}
\definecolor{bgreen}{rgb}{0.2, 0.6, 0.4}
\definecolor{bred}{rgb}{0.8, 0.4, 0.2}
\definecolor{bviolet}{rgb}{0.7, 0.2, 0.7}
\definecolor{blackred}{rgb}{0.6, 0.3, 0.3}
\definecolor{blackblue}{rgb}{0.3, 0.3, 0.6}

\usetikzlibrary{decorations.markings}
\usetikzlibrary{shapes.geometric}

\pgfdeclarelayer{edgelayer}
\pgfdeclarelayer{nodelayer}
\pgfsetlayers{edgelayer,nodelayer,main}

\tikzstyle{none}=[inner sep=0pt]
\definecolor{hexcolor0x0c07e9}{rgb}{0.027,0.502,0.914}
\definecolor{hexcolor0x190707}{rgb}{0.098,0.027,0.027}
\definecolor{Black}{rgb}{0.098,0.027,0.027}

\tikzstyle{bn}=[circle,fill=hexcolor0x0c07e9,draw=Black,line width=0.8 pt]
\tikzstyle{gn}=[circle,fill=Lime,draw=Black,line width=0.8 pt]
\tikzstyle{kn}=[circle,fill=hexcolor0x190707,draw=Black]

\tikzstyle{arc}=[->,draw=Black,line width=1.400]
\tikzstyle{arc2}=[-,draw=Black,postaction={decorate},decoration={markings,mark=at position .5 with {\arrow{>}}},line width=1.400]
\tikzstyle{edge}=[-,draw=Black,line width=1.400]

\begin{document}

\maketitle

    \begin{abstract}
In this paper, we  count acyclic and strongly connected
  uniform directed labeled  hypergraphs. For these combinatorial
  structures, we introduce a specific
  generating function allowing us to
  recover  and generalize some results 
  on the number of directed acyclic graphs and 
  the number of strongly connected directed graphs.
    \end{abstract}

    \section{Introduction}
A  \textit{directed graph} or \textit{digraph} consists of a finite node set $\mathcal{V}$ with a subset $\mathcal{E}$ of
$\mathcal{V} \times \mathcal{V}$ (the arcs) and we do not allow neither loops nor multiple arcs.

In the seventies, several researchers including
Liskovets~\cite{Liskovets69a,Liskovets69b}, Robinson~\cite{Robinson70,Robinson73}, Stanley~\cite{Stanley73}
or Wright~\cite{Wright1971} studied enumerative aspects of important families of digraphs including
\textit{Directed Acyclic Graphs} (DAGs) or \textit{strongly connected digraphs}.

A hypergraph is a generalization of a graph in which an (hyper)edge can join any number of nodes.
Hypergraphs have been extensively studied~\cite{Berge73,Berge89} as they are
 very useful  to model
concepts and structures in various aspects of Computer Science (combinatorial optimization,
algorithmic game theory, machine learning, constraint satisfaction problem,
data mining and indexing, ...).

In this paper, we deal with directed hypergraphs or  simply \textit{dihypergraphs}
 (also known as \textit{And/Or graphs}~\cite{Martelli73,Levi76,Gallo1993}).
As far as we know, these objects have been introduced in the Computer Science
literature by Boley as a representation language~\cite{Boley}. For detailed surveys on directed hypergraphs,
 algorithms and applications, we refer the reader to the papers of
Gallo, Longo, Pallotino and Nguyen~\cite{Gallo1993} and of Ausiello and Luigi~\cite{Ausiello}.
Following the recent enumerative results on digraphs
of de Panafieu and Dovgal~\cite{Dovgal2019}, Archer, Gessel, Graves and Liang~\cite{Gessel2019},
 our aim in this article is to study enumerative aspects of some
 families of dihypergraphs.
 
\section{Definitions}

Terminology for   dihypergraphs  is established
in the book of Harary,   Norman  and \\ Cartwright~\cite{Harary1965}  or
in the  paper of Gallo, Longo, Pallottino and  Nguyen~\cite{Gallo1993}.

A {directed (labeled) hypergraph}  (or simply {dihypergraph})
$\mathcal H$ is a pair $(\mathcal V, \mathcal E)$ where $\mathcal V$
is a non-empty finite set of {nodes} and $\mathcal E$ is a set
of ordered  pairs of non-empty subsets of $\mathcal V$  called
\emph{directed hyperedges } (or \emph{hyperarcs}).
That means a hyperarc  $e$ is an ordered pair
$(T(e), H(e))$, of  disjoint subsets of $\mathcal V$ such that $T(e)
\neq \emptyset$, $H(e)\neq \emptyset$.  $T(e)$ is called the
\emph{tail}  of  the hyperarc $e$  while $H(e)$ is its \emph{head}.

A dihypergraph $\mathcal H = (\mathcal V,\mathcal E)$  is called
\emph{$b$-uniform} iff  for any $e\in \mathcal E$, $|T(e)| + |H(e)|=b$
(that is all hyperarcs are built with  the same number of
nodes). Clearly, the $2$-uniform dihypergraph is  the standard
digraph. The dihypergraph $(\emptyset, \emptyset)$ is called the empty
dihypergraph.

A \emph{directed path} (or  \emph{path}) $P_{st}$ of length $\ell$
in a dihypergraph $\mathcal H = (\mathcal V,\mathcal E)$, is a sequence
of  nodes and  hyperarcs
$P_{st}=\left(v_1=s,e_{i_1},v_2,\ldots,v_\ell,e_{i_{\ell}},v_{\ell+1}=t\right)$
where:
\[
  s \in T(e_{i_1}),\quad   t \in T(e_{i_\ell}) \mbox{ and }
  v_j\in T(e_{i_{j-1}})\cap H(e_{i_j}) \mbox{ for } j=2..\ell.
\]
Nodes  $s$ and $t$ are respectively the \emph{origin}
and the \emph{destination} of the path $P_{st}$
and we say that $t$ is connected to $s$.
The path $P_{st}$ is said  \emph{simple}  if all nodes on the path 
are distincts except possibly the origin $s$ and the destination $t$. A
\emph{directed cycle} (or simply \emph{cycle}) in  a dihypergraph is a
path where the 
origin and the destination coincide. 
A dihypergraph is said \emph{acyclic} iff it has no cycle.

Given a dihypergraph $\mathcal H = (\mathcal V, \mathcal E)$,
we define the relation $\mathcal  R$ on $\mathcal V$
by $u \, \mathcal R \, v$ if there is a
(directed) path
from $u$ to $v$ in $\mathcal H$ and vice versa.
 It is easy to show that $\mathcal R$ is an
equivalence relation on $\mathcal V$.
The equivalence classes are called the \emph{strongly connected components} of
$\mathcal H$.  
A dihypergraph is \emph{strongly connected} (or simply \emph{strong})
if it has a unique strong component.


According to Robinson~\cite{Robinson70, Robinson73}
an \emph{out-component} of a digraph is a strong 
component which cannot be reached from any other
strong component. Such a component is  called 
\emph{source strong  component} by Gessel~\cite{Gessel2019}
and \emph{source-like strong
  connected component} by de Panafieu and Dovgal~\cite{Dovgal2019}.
A source (strong) component is called simply a \emph{source} if it contains
exactly one node.

Obviously, we have  the following  Lemma.
\begin{lem}\label{lem:out-component}
  Every non-empty dihypergraph has  at least a source strong  component.
\end{lem}

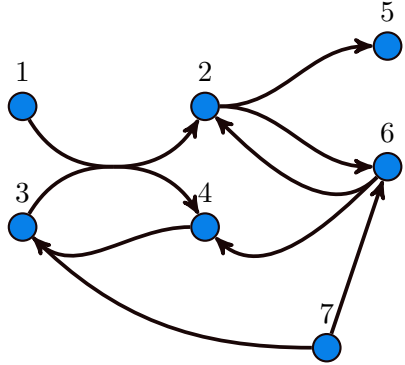
\begin{figure}[h]
  \begin{center}
\begin{tikzpicture}[>=stealth',thick, scale = 0.8]
  \begin{pgfonlayer}{nodelayer}
    \node [style=bn, label = {$1$}] (0) at (-17, 7) {};
    \node [style=bn, label = {$2$}] (1) at (-14, 7) {};
    \node [style=bn, label = {$3$}] (2) at (-17, 5) {};
    \node [style=bn, label = {$4$}] (3) at (-14, 5) {};
    \node [style=bn, label = {$5$}] (4) at (-11, 8) {};
    \node [style=bn, label = {$6$}] (5) at (-11, 6) {};
    \node [style=bn, label = {$7$}] (6) at (-12, 3) {};
    \node [style=none] (7) at (-15.5, 6) {};
  \end{pgfonlayer}
  \begin{pgfonlayer}{edgelayer}
    \draw [style=edge, bend right] (0) to (7.center);
    \draw [style=edge, bend left] (2) to (7.center);
    \draw [style=arc, bend right] (7.center) to (1);
    \draw [style=arc, bend left] (7.center) to (3);
    \draw [style=arc] (1) to  [out = 0, in = 180] (4);
    \draw [style=arc] (1) to [out = 0, in = 180] (5);
    \draw [style=arc] (5) to [out = 225, in = -45] (3);
    \draw [style=arc] (5) to [out = 225, in = -45] (1);
    \draw [style=arc] (6) to (5);
    \draw [style=arc] (3) to [out = 180, in = -45](2);
    \draw [style=edge] (6) to [out = 180, in = -45 ] (2);
  \end{pgfonlayer}
\end{tikzpicture}            
  \end{center}
  \caption{A general directed hypergraph with nodes
    $\{1,\,2,\,3,\,4,\,5,\,6,\,7\}$ 
    built with $5$ hyperarcs
    $\{1,3\} \rightarrow \{2,4\}$, 
    $\{2\} \rightarrow \{5,6\}$,
    $\{6\} \rightarrow \{2,4\}$,
    $\{7\} \rightarrow \{6\}$,
    $\{7,4\} \rightarrow \{3\}$, 
  The subset of nodes
  $\{2,6,4,3\}$ forms a directed cycle. }
\end{figure}

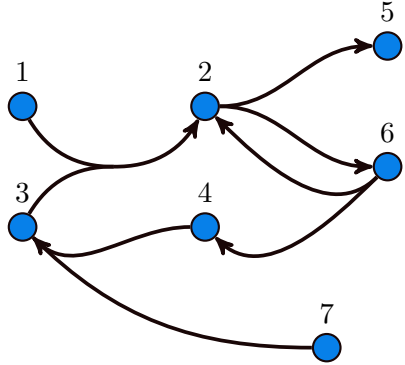
\begin{figure}[h]
  \begin{center}
\begin{tikzpicture}[>=stealth',thick, scale = 0.8]
  \begin{pgfonlayer}{nodelayer}
    \node [style=bn, label = {$1$}] (0) at (-17, 7) {};
    \node [style=bn, label = {$2$}] (1) at (-14, 7) {};
    \node [style=bn, label = {$3$}] (2) at (-17, 5) {};
    \node [style=bn, label = {$4$}] (3) at (-14, 5) {};
    \node [style=bn, label = {$5$}] (4) at (-11, 8) {};
    \node [style=bn, label = {$6$}] (5) at (-11, 6) {};
    \node [style=bn, label = {$7$}] (6) at (-12, 3) {};
    \node [style=none] (7) at (-15.5, 6) {};
  \end{pgfonlayer}
  \begin{pgfonlayer}{edgelayer}
    \draw [style=edge, bend right] (0) to (7.center);
    \draw [style=edge, bend left] (2) to (7.center);
    \draw [style=arc, bend right] (7.center) to (1);
    \draw [style=arc] (1) to  [out = 0, in = 180] (4);
    \draw [style=arc] (1) to [out = 0, in = 180] (5);
    \draw [style=arc] (5) to [out = 225, in = -45] (3);
    \draw [style=arc] (5) to [out = 225, in = -45] (1);
    \draw [style=arc] (3) to [out = 180, in = -45](2);
    \draw [style=edge] (6) to [out = 180, in = -45 ] (2);
  \end{pgfonlayer}
\end{tikzpicture}            
  \end{center}
\caption{
  A  $3$-uniform directed hypergraph
  with nodes $\{1,\,2,\,3,\,4,\,5,\,6,\,7\}$
  and 4  hyperarcs
    $\{1,3\} \rightarrow \{2\}$, 
    $\{2\} \rightarrow \{5,6\}$,
    $\{6\} \rightarrow \{2,4\}$,
    $\{7,4\} \rightarrow \{3\}$. 
}
\end{figure}

Throughout the rest of this paper, a dihypergraph is
a $b$-uniform directed hypergraph. Similarly a hyperarc
 with $b$ nodes is called simply a hyperarc. 
 Graphs, digraphs or dihypergraphs are labeled.

\section{Hypergraphic genenerating functions}

We introduce a
new type of generating function called \emph{hypergraphic generating
  function}  defined as follow. 
The variables $x$ and $y$ are reserved to mark 
nodes and  hyperarcs. 
\begin{dfn}
The hypergraphic generating function (or simply HGF) 
for the sequence $(f_n(y))_{n\geq 0}$ is defined by
\begin{equation}\label{eq:F_xy}
  F(x,y) := \sum_{n=0}^\infty
  \frac{f_n(y)}{(1+y)^{\binom{n}{b}}}\frac{x^n}{n!},
\end{equation}
where $b\geq 2$.
\end{dfn}

Our hypergraphic generating function is a generalization of
the \emph{graphic generating function} (GGF) introduced by
Read~\cite{Read1960} and Robinson~\cite{Robinson73}.
In particular, the \emph{special generating function} of
Robinson~\cite{Robinson73}  corresponds to the  case $b=2$ and  $y=1$
and the  graphical generating function  corresponds to the case $b=2$.
Graphic generating functions are very useful as shown by the results
of Bender, Richmond,  Robinson and  Wormald~\cite{Bender1986},
of Gessel~\cite{Gessel1995}, of Gessel and  Sagan~\cite{Gessel1996},
and very recently of Archer, Gessel, Graves and Liang~\cite{Gessel2019} and
de Panafieu and  Dovgal~\cite{Dovgal2019}.

For convenience, given a family of dihypergraphs $\mathcal F$
enumerated by the sequence $(f_n(y))_{n\geq 0}$, the 
\emph{exponential generating function} (EGF) will be denoted
by
\begin{equation}\label{eq:f_xy}
  f(x,y) := \sum_{n=0}^\infty f_n(y)\frac{x^n}{n!},
\end{equation}
and its HGF by \eqref{eq:F_xy}.
As some additionnal
variables may be added for specific parameters,
 we often use  \emph{multivariate generating functions} 
 (see~Flajolet and Sedgewick~\cite[Definition~III.4]{FSBook} for
 \emph{multi-index convention}).
\begin{dfn}
The exponential  multivariate
generating function 
of a family $\mathcal F$ will be denoted by
\[
  f(x,y, u) = \sum_{n,p}f_{n, p}(y)
   u^{ p} \frac{x^n}{n!},
\]
and the corresponding multivariate hypergraphic generating function
 is
\[
  F(x,y, u) = \sum_{n,p}\frac{f_{n,
      p}(y)}{(1+y)^{\binom{n}{b}}} 
   u^{p} \frac{x^n}{n!},
\]
where $u$ the variable for some source component.
Throughout this paper, the quantities $f(x,y,1)$ and
  $F(x,y, 1)$   coincide with $f(x,y)$ and $F(x,y)$ respectively.
\end{dfn}

We observe that the HGF is obtained by dividing the coefficient of
$n! \, x^n$ in the EGF by $(1+y)^{\binom{n}{b}}$.
This linear operation is named by Robinson~\cite{Robinson73}
 as $\Delta$ for the case $b=2$ and $y=1$. We can  use similar
 notation to  convert an EGF to a HGF of family of dihypergraphs
 $\mathcal F$.
 \begin{dfn}\label{def:delta}
   Let $\mathcal F$ be a family of dihypergraphs with EGF $f$ and HGF $F$.
   We define  $\Delta_{y,b}$ as the linear operator on generating
   functions which transform $f$ into $F$ :
   \begin{equation}\label{Delta}
   F(x,y) = \Delta_{y,b} \left(f(x,y) \right)\, .
 \end{equation}
\end{dfn}
 \noindent
 Let us remark that the operator  $\Delta_{y,b}$ acts only w.r.t. the
 variable $x$. As an example of using $\Delta_{y,b}$, consider 
 all sets of empty dihypergraphs (dihypergraph  that
 contains no hyperarc). The EGF of such graphs  is
 $\sum_{n\geq 0} x^n/n!$ and then the associated HGF is 
 \begin{equation}\label{eq:theta_y}
   \begin{split}
    \theta_b(x,y) &:= \Delta_{y,b}\left(\sum_{n\geq
        0}\frac{x^n}{n!}\right) ,\\ & =
    \sum_{n\geq 0}^\infty
    \frac{1}{(1+y)^{\binom{n}{b}}} \frac{x^n}{n!}.
    \end{split}
  \end{equation}
Observe that 
de Panafieu and Dovgal~\cite{Dovgal2019} used the \emph{exponential
      Hadamard product} to convert an EGF to a graphic generating
    function when working on digraphs. Such operation is simply defined below.
\begin{dfn}
    The exponential Hadamard product
    of $f(x)=\sum_{ n \geq 0}f_n \frac{x^n}{n!}$  and
    $g(x)=\sum_{ n \geq 0}g_n \frac{x^n}{n!}$ is the exponential
    generating functions of the sequence $(f_ng_n)_{n\geq 0}$.
    It is denoted
    $f(x)\odot g(x)$ and we have
    \[
      f(x)\odot g(x) = \sum_{ n \geq 0}f_ng_n \frac{x^n}{n!}\,.
    \]
  \end{dfn}
Then,  given a family of dihypergraphs $\mathcal F$ with EGF $f$
         and HGF $F$, the linear operator  $\Delta_{y,b}$ and the
         exponential  Hadamard 
  product are linked by the equation
         \[
           F(x,y) = \theta_b(x,y)\odot f(x,y) = \Delta_{y,b}\left(f(x,y)\right),
         \]
         where $\theta_b(x,y)$ is the HGF defined
         by~\eqref{eq:theta_y}. 

      Now, we introduce the \emph{arrow product}
      which already appears  in~\cite{Robinson73,
        Robinson95,Gessel1996b}. 
      The definition of the arrow product
      of two  families  of digraphs  $\mathcal A$ and $\mathcal B$
      viewed as symbolic methods  is
      defined explicitly in~\cite{Dovgal2019}. Such definition is
      extended here to dihypergraphs.
      \begin{dfn}\label{def:arrow_product}
        The arrow product $\mathcal{C}$ of two  families of dihypergraphs
        $\mathcal A$ and $\mathcal B$ is the
        family that consists in pairs $(A,B)$ with
        $A\in \mathcal A$ and $B\in \mathcal B$
        relabeled so that objects $A$ and $B$ have disjoint labels and
        where an arbitrary number of hyperarcs
        have  their tails belonging to  $A$  and
        their heads belong to $B$ (Fig.~\ref{fig:arrow_product}). 
      \end{dfn}

      \begin{figure}[h!]
        \begin{center}
          \begin{tikzpicture}[>=stealth',thick, scale = 0.8]
	\begin{pgfonlayer}{nodelayer}
          \node [style=bn] (0) at (-15, 7) {};
		\node [style=bn] (1) at (-15, 5) {};
		\node [style=bn] (2) at (-17, 6) {};
		\node [style=bn] (3) at (-12, 8) {};
		\node [style=bn] (4) at (-11, 6) {};
		\node [style=bn] (5) at (-9, 8) {};
		\node [style=bn] (6) at (-12, 3) {};
		\node [style=bn] (7) at (-9, 4) {};
		\node [style=none] (8) at (-13, 5) {};
		\node [style=none] (9) at (-9.25, 6.25) {};
		\node [style=none] (10) at (-13, 6) {};
		\node [style=bn] (11) at (-17, 4) {};
		\node [style=bn] (12) at (-11, 4) {};
                \node [style=none] (13) at (-18, 9) {};
		\node [style=none] (14) at (-14, 2.5) {};
                \node [style=none] (15) at (-13, 9) {};
		\node [style=none] (16) at (-8, 2.5) {};
                \node [style=none] (17) at (-16, 2) {$A \in \mathcal
                  A$};
                \node [style=none] (17) at (-10.5, 2) {$B \in \mathcal
                  B$};
                \node [style=bn] (18) at (-17, 8) {};
	\end{pgfonlayer}
	\begin{pgfonlayer}{edgelayer}
		\draw [style=arc] (2) to [out = 0, in = 180] (0);
		\draw [style=arc] (2) to [out = 0, in = 180] (1);
		\draw [style=arc] (3) to [out = 0] (4);
		\draw [style=arc] (3) to [out = 0] (5);
		\draw [style=edge, bend right,color = red] (1) to (8.center);
		\draw [style=arc,color = red] (8.center) to [in = -135] (4);
		\draw [style=arc,color = red] (10.center) to (4);
		\draw [style=edge, color = red] (0) to [out=-45,in =
                180] (10.center); 
		\draw [style=edge, color = red] (1) to [in = 180] (10.center);
		\draw [style=arc,color = red] (0) to (3);
		\draw [style=edge, bend left] (5) to (9.center);
		\draw [style=arc, bend left] (9.center) to (4);
		\draw [style=arc, bend right] (9.center) to (7);
                \draw [style=arc, bend left] (1) to (11);
                \draw [style = dashed, rounded corners = 5mm,  very
                thick, color = bgreen](13)rectangle (14);
                \draw [style = dashed, rounded corners = 5mm,  very
                thick, color = bgreen](15)rectangle (16);  
	\end{pgfonlayer}
      \end{tikzpicture}
        \end{center}
        \caption{Arrow product for  dihypergraphs}
        \label{fig:arrow_product}
    \end{figure}
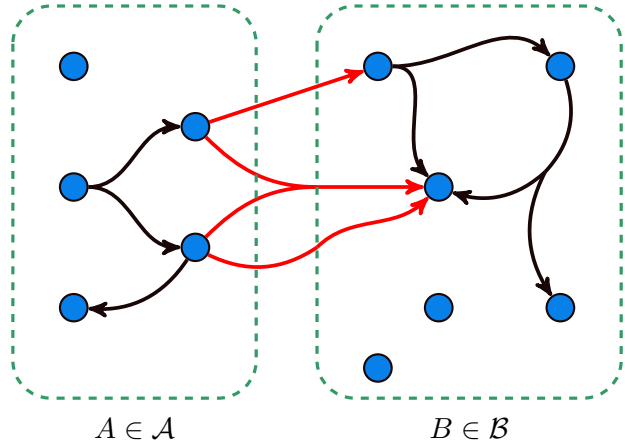
    
The following  lemmas  extend on dihypergraphs some results on
symbolic methods of EGFs (cf. Flajolet and Sedgewick~\cite{FSBook})
and  symbolic methods of GGFs as introduced by
de Panafieu and Dovgal~\cite{Dovgal2019}. 

\begin{lem}\label{lem:sum}
  Given two families $\mathcal F$ and $\mathcal G$  of
  dihypergraphs   with   HGFs   $F(x,y)$ and $G(x,y)$,
  the HGF of the disjoint union $\mathcal F+ \mathcal G$ is
  \[
    F(x,y) +  G(x,y),
  \]
  where $x$ and $y$ mark respectively nodes and hyperarcs.
\end{lem}

\begin{lem}\label{lem:subs}
  Given a  family of dihypergraphs $\mathcal F$ with HGF  $F$,
  if a variable $u$ marks the
  number of some family of source components in the HGF $F(x,y,u)$ the 
  HGF for the elements of $\mathcal F$  which have
  a distinguished subset of source components is $F(x, y, u + 1)$
  where $x$ and $y$ mark respectively nodes and hyperarcs.
\end{lem}

The proofs of   Lemmas~\ref{lem:sum}~and~\ref{lem:subs}
are elementary by means of symbolic methods on
EGFs and Definitions~\ref{def:delta}~and~\ref{def:arrow_product}.

As an example of using the parameter $u$ for a family of
dihypergraphs $\mathcal F$, we may use $u$ to mark
the number of sources in the HGF $F(x,y,u)$.
Then, $F(x,y,1)$ is the HGF of the whole family $\mathcal
F$ without  distinguishing  
 if a node is a source or not and $F(x,y,0)$ is the HGF of
 dihypergraphs in $\mathcal F$ without any source.

\begin{remark}
The substitution of $u$ by $u+1$ means
that  items are marked or left unmarked. Conversely, replacing $u$ with $u-1$
corresponds to an inclusion-exclusion principle.
\end{remark}

\begin{lem}
  Let $F(x, y)$ and $G(x, y)$ be the HGFs of two families
  of dihypergraphs $\mathcal F$ and $\mathcal G$.
  The HGF of the arrow product (cf.~Definition~\ref{def:arrow_product})
  of the families $\mathcal F$  and $\mathcal G$
  is equal to $F(x, y)\,G(x, y)$.
\end{lem}

\begin{proof}
Let $(f_n(y))$ and $(g_n(y))$ be the associated sequences of the two
families $\mathcal F$  and $\mathcal G$. Then, the sequence
associated to the HGFs $F(x,y)G(x,y)$ is
\begin{align*}
  c_n(y)&=(1+y)^{\binom{n}{b}}  n![x^n]\left(
\sum_{k\geq 0}    \frac{f_k(y)}{(1+y)^{\binom{k}{b}}}\frac{x^k}{k!}\right)\left(
\sum_{\ell\geq 0}     \frac{g_\ell(y)}{(1+y)^{\binom{\ell}{b}}}\frac{x^\ell}{\ell!}
          \right),\\
  & = \sum_{k=0}^n\binom{n}{k}(1+y)^{\binom{n}{b}-\binom{k}{b} -
    \binom{n-k}{b}}f_k(y)g_{n-k}(y). 
\end{align*}
\end{proof}

There is a direct combinatorial explanation
for the exponent $\binom{n}{b} - \binom{k}{b} - \binom{n-k}{b}$.
Consider two dihypergraphs $F$ and $G$ of sizes $k$ and $n-k$,
and their arrow product $H$ (of size $n$).
$F$ and $G$ are combined and relabeled.
Any of the $\binom{n}{b}$ possible sets of nodes
can become a hyperarc from the arrow product,
except the $\binom{k}{b}$ sets that contain only nodes from $F$,
and the $\binom{n-k}{b}$ sets that contain only nodes from $G$. 
We can also use  the Vandermonde's identity for any nonnegative
integers $b, m, n$: $$    {{m+n \choose b}=\sum _{k=0}^{b}{m \choose
    k}{n \choose b-k}},$$ 
to show that
\[
  \binom{n}{b}-\binom{k}{b} - \binom{n-k}{b} =
  \sum_{i+j=b,i,j>0}{k \choose i}{n -k \choose j} \,.
\]

\begin{lem}\label{lem:total_hyperedge}
  The total number of   hyperarcs on $n$ nodes is equal to
  \[
    (2^b-2)\binom{n}{b},\quad \mbox{    for $b\geq 2$.}
  \]
\end{lem}
  
  \begin{proof}
    The   number of hyperarcs with exactly $k$ tails ($0<k< b$)
    and $b-k$ heads  is equal    to
\[
  \binom{n}{k}\binom{n-k}{b-k}.
\]
Summing over $k$, we have
    \[
 \sum_{k=1}^{b-1} \binom{n}{k}\binom{n-k}{b-k} = (2^b-2)\binom{n}{b}.
    \]
  \end{proof}

  \begin{lem} \label{lem:EGF_HGF}
    The EGF of all dihypergraphs $h(x, y)$ is
    \begin{equation}
    h(x,y) = \sum_{n\geq 0}^\infty
    (1+y)^{(2^{b}-2)\binom{n}{b}}\frac{x^n}{n!}. 
    \end{equation}
  The HGF of all dihypergraphs $H(x, y)$ is
  \begin{equation}\label{eq:Hxy}
    H(x,y) = \sum_{n\geq 0}^\infty
    (1+y)^{(2^{b}-3)\binom{n}{b}}\frac{x^n}{n!}. 
    \end{equation}
  \end{lem}
  
  \begin{proof} The proof is obvious from the definition of the HGFs and
    by Lemma~\ref{lem:total_hyperedge}.
    \end{proof}
\section{Acyclic or strong  dihypergraphs} 
In this Section, we give exact enumerations of acyclic
or strongly 
dihypergraphs.
Our results extend those
in~\cite{Stanley73,Robinson73,Gessel1996, Robinson77,Dovgal2019}
on enumeration of  these families in digraphs to dihypergraphs.
We notice also that
a different approach has been given
by Ostroff~\cite{PhD2013} to count strong digraphs

Let us recall that Robinson~\cite[Corollary~ 1]{Robinson73}
showed that the counting sequence $\alpha_ n(y)$
of acyclic digraphs on $n$ nodes
satisfies
\[
  \sum_{n=0}^\infty \frac{\alpha_n(y)}{(1+y)^{\binom{n}{2}}}\frac{x^n}{n!} =
 \left(  \sum_{n=0}^\infty \frac{(-1)^n}{(1+y)^{\binom{n}{2}}}
    \frac{x^n}{n!}\right)^{-1}.
\]
Theorem~\ref{THM:ACYCLIC2} generalizes this identity for
dihypergraphs.  Let us  define the
 HGF of the sequence $((-1)^n)_{n\geq 0}$ denoted $\phi(x,y)$. We have
\begin{equation}\label{eq:EGF-phi}
\phi(x,y):=\sum_{n=0}^{\infty}\frac{(-x)^n}{n! \, (1+y)^{{n \choose b}}} \,.
\end{equation}

\begin{thm}\label{THM:ACYCLIC2}
Let $a_n(y) = \sum_{q=0}^{\binom{n}{b}} a_{n,q} y^q$ be the counting
sequence of acyclic dihypergraphs 
where $a_{n,q}$ denotes  the number of acyclic dihypergraphs with $n$
nodes and $q$ hyperarcs, and
$ A(x,y) = \sum_{n=0}^\infty
\frac{a_n(y)x^n}{n!(1+y)^{\binom{n}{b}}}$ be its associated HGF.
$A(x,\,y)$ satisfies
  \begin{equation}\label{eq:identity}
    A(x,y) = \phi(x,y)^{-1} ,
  \end{equation}
  where    $\phi$ is defined by~\eqref{eq:EGF-phi}.
\end{thm}

\begin{proof}
  
  Let $u$ be  the variable marking  the number of sources 
  in the EGF or  HGF of all acyclic dihypergraphs $A(x,y,u)$.
  By the Lemma~\ref{lem:subs},
  the  HGF for the dihypergraphs where each source node   
  is either marked, or
  left unmarked by the variable $u$ is $A(x,y,u+1)$.
    Next, the EGF of a set of isolated nodes 
  is $\exp(ux)$ (dihypergraph without any hyperarc)
  and so the associated HGF is $\Delta_{y,b}\left(\exp(ux)\right)$.
We observe that  an acyclic dihypergraph with some marked sources can 
be viewed  as an arrow product of a set of  nodes
(the marked sources) with an acyclic
dihypergraph. This decomposition
  implies
\[
  A(x,y,u+1)=\Delta_{y,b}\left(\exp(ux) \right)\, \times  \, A(x,y) \, .
\]
Substituting $u$ by $-1$ leads to $A(x,y,0) = 1$ (the only acycic dihypergraph
without a source  is the empty dihypergraph). Since $\Delta_{y,b}\left(\exp(ux) \right)
= \phi(x,y)$ where $\phi$ is given by~\eqref{eq:EGF-phi}, we get the result.
\end{proof}

\begin{remark}
  Similar proof can be obtained using first the inclusion-exclusion
  principle to get    $a_0(y)=1$ and for $n\geq 1$
  \begin{equation}\label{eq:number_acyclic_rec1}
     a_n (y)= \sum_{k=1}^n (-1)^{k-1} \binom{n}{k}
     (1+y)^{\binom{n}{b}-\binom{k}{b}-\binom{n-k}{b}} a_{n-k}(y) \, ,
   \end{equation}
   which can be rewritten as
   \begin{equation}\label{eq:number_acyclic_rec2}
     \sum_{k=0}^n(-1)^{n-k}
     \binom{n}{k}(1+y)^{\binom{n}{b}-\binom{k}{b}-\binom{n-k}{b}}a_{k}=\delta_{n0},
   \end{equation}
   where $\delta_{n0}$ is Kronecker's symbol, and then by checking  
   that $A(x,y)\phi(x,y)=1$. \\
   In term of $n$, an explicit expression of
$a_n(y)$ can be obtained from the identity $A(x,y)\phi(x,y)=1$:
\[
     a_n(y) = \sum_{j \geq 0} (-1)^j \sum_{n_1 + \cdots + n_j = n}
     \binom{n}{n_1, \ldots, n_j} (1+y)^{\binom{n}{b} - \sum_{i=1}^j
       \binom{n_i}{b}} 
\]
\end{remark}
  \begin{thm}\label{theo:hgf}
  Let $S$ be the set of all strongly connected dihypergraphs, if
  $s$ is the  associated EGF, then the HGF
  of all dihypergraphs  defined by   \eqref{eq:Hxy} and
  the EGF  $s(x,y)$ verify
  \begin{equation}\label{eq:strong_tout}
    H(x,y) = \left(\Delta_{y,b} \left(\exp\left(-s(x,y)\right)\right) \right)^{-1}.
  \end{equation}
\end{thm}

\begin{proof}
  Let  $u$  be a variable  marking the number of 
  strongly connected components which 
  are source components~(see Lemma~\ref{lem:out-component}) in the EGF or 
  in the HGF $H(x,y,u)$ of all
  dihypergraphs.    By the Lemma~\ref{lem:subs},
  the  HGF for the dihypergraphs where each source strong 
  component is either marked, or
  left unmarked by the variable $u$ is $H(x,y,u+1)$.
  Next, the EGF of the set of  strongly connected components
  is $\exp(u\, s(x,y))$ and so the associated HGF is
  $\Delta_{y,b}\left( \exp(u\, s(x,y)) \right)$.
  We observe that a  dihypergraph with some marked source components
  can be viewed  as an arrow product of a set of  strong
  dihypergraphs
  (the marked source components)  with a dihypergraph. This
  decomposition   implies
\[
  H(x,y,u+1)=\Delta_{y,b}\left(\exp(u\, s(x,y))\right)\times H(x,y) \, .
\]
Then replacing $u$ with $-1$  gives the result since  $H(x,y,0) = 1$
(the only dihypergraph
without a source component is the empty dihypergraph).
\end{proof}

\begin{remark}
  Notice also that  Theorem~\ref{theo:hgf} leads to  a recursive
  relation satisfied by $(s_n(y))$ where
   $s_n (y)= n![x^n]s(x,y)$ with $s(x,y)$ is the EGF of
   all strongly connected dihypergraphs. Following  the same techniques
   using by Robinson in~\cite[Section~4.]{Robinson73}, we can easily
   show that $s_0(y)=1$ and
    \[
      s_n(y) = \lambda_n(y) + \sum_{t=1}^{n-1}\binom{n-1}{t}s_{n-t}(y)\lambda_t(y) \, ,
    \]
  with
  \[
  \lambda_n(y) = (1+y)^{(2^{b}-2)\binom{n}{b}} -\sum_{t=1}^{n-1}
  \binom{n}{t} (1+y)^{(2^{b}-2)\binom{t}{b}}\lambda_{n-1}(y) \, .
\] 
\end{remark}

\section{Conclusion}

 Our paper deal with directed uniform hypergraphs by introducing a
  specific type of generating functions to obtain
  generating functions of acyclic and strong  dihypergraphs.
  We think that many families of dihypergraphs can be enumerated
  using the same
  methods. More generally, what is the most general model of
  graph-like objects 
  where DAGs and strongly connected components can be defined
  and counted following the same techniques?
  
  In future works, it would be interesting to compute the asymptotic
  number of these 
  combinatorial structures (as in~\cite{Bender1986} for dense digraphs
  and in~\cite{Dovgal2020, NVS} for sparse random digraphs) and to study the
   appearance of strongly connected components
   (as in ~\cite{Karp,Luczak}) during some random dihypergraphs
   processes. For example, 
   when enriching the structures by adding hyperarc one by one,
   how many hyperarcs are needed to have asymptotically
   almost surely structures containing complex strong components?
\nocite{*}
\bibliographystyle{plain} 
\bibliography{dihypergraph}

\end{document}